\documentclass[a4paper,12pt]{article}
\usepackage{epsfig}
\usepackage{amsmath,amsthm,enumerate}
\usepackage{amsfonts}
\usepackage{latexsym}
\usepackage{amsbsy}
\usepackage{amssymb}
\numberwithin{equation}{section}

\usepackage{amsfonts,amssymb,amscd,amsmath,enumerate,verbatim,calc,eucal}

\newcommand{\ma}{\mathcal}
\newcommand{\mf}{\mathfrak}
\newcommand{\m}{\CMcal}

\theoremstyle{plain}
\newtheorem{theorem}{Theorem}[section]

\newtheorem{lemma}[theorem]{Lemma}
\newtheorem{proposition}[theorem]{Proposition}

\newtheorem{definition}[theorem]{Definition}

\begin{document}

\begin{center}
 \Large\bf{A covariant Stinespring type theorem for $\tau$-maps}
\end{center}

\vspace{0.25cm}

\begin{center} HARSH TRIVEDI
\end{center}

\vspace{0.25cm}

     \begin{abstract}
Let $\tau$ be a linear map from a unital $C^*$-algebra $\m A$ to a von Neumann algebra $\m B$ and 
let $\m C$ be a unital $C^*$-algebra. 
A map $T$ from a Hilbert $\m A$-module $E$ to a von Neumann $\m C$-$\m B$ module $F$ is called a $\tau$-map
if
$$\langle T(x),T(y)\rangle=\tau(\langle x, y\rangle)~\mbox{for
all}~x,y\in E.$$ A Stinespring type theorem for $\tau$-maps and its
covariant version are obtained when $\tau$ is completely positive. We show that
there is a bijective correspondence between the set of all 
$\tau$-maps from $E$ to $F$ which are $(u',u)$-covariant with respect to a dynamical system
$(G,\eta,E)$ and the set of all $(u',u)$-covariant $\widetilde{\tau}$-maps from
the crossed product $E\times_{\eta} G$ to $F$, where $\tau$ and $\widetilde{\tau}$ are
completely positive.

\vspace{0.5cm}

  \end{abstract}

\vspace{0.5cm}

\section{Introduction}

A linear mapping $\tau$ from a (pre-)$C^*$-algebra $\m A$ to a (pre-)$C^*$-algebra $\m B$ is called {\it completely positive} if
$$\sum_{i,j=1}^n b_j^{*}
\tau(a_j^{*}a_i)b_i\geq 0$$
for each $n\in \mathbb{N}$, $b_1,b_2,\ldots,b_n\in\m
B$ and $a_1,a_2,\ldots,a_n\in\m A$. The completely positive maps are used significantly in the theory of measurements, quantum mechanics, 
operator algebras etc. Paschke's Gelfand-Naimark-Segal (GNS)
construction (cf. Theorem 5.2, \cite{Pas73}) characterizes completely positive maps between unital $C^*$-algebras, which is an abstraction of
the Stinespring's theorem for operator valued completely 
positive maps (cf. Theorem 1, \cite{St55}). Now we define Hilbert $C^*$-modules which are a generalization of Hilbert spaces and $C^*$-algebras, were introduced
by Paschke in the paper mentioned above and were also studied independently by Rieffel in \cite{Ri74}. 
\begin{definition}
Let $\m B$ be a (pre-)$C^*$-algebra and ${E}$ be a vector space
which is a right $\m B$-module satisfying $\alpha(xb)=(\alpha
x)b=x(\alpha b)$ for $x\in {E},b\in \m B,\alpha\in\mathbb{C}$. The
space ${E}$ is called an {\rm inner-product $\m B$-module} or a {\rm
pre-Hilbert $\m B$-module} if there exists a mapping $\langle
\cdot,\cdot \rangle : E \times E \to \m{B}$ such that
\begin{itemize}
 \item [(i)] $\langle x,x \rangle \geq 0 ~\mbox{for}~ x \in {E}  $ and $\langle x,x \rangle = 0$ only if $x = 0 ,$
\item [(ii)] $\langle x,yb \rangle = \langle x,y \rangle b ~\mbox{for}~ x,y \in {E}$ and  $~\mbox{for}~ b\in \m B,  $
 \item [(iii)]$\langle x,y \rangle=\langle y,x \rangle ^*~\mbox{for}~ x ,y\in {E} ,$
\item [(iv)]$\langle x,\mu y+\nu z \rangle = \mu \langle x,y \rangle +\nu \langle x,z \rangle ~\mbox{for}~ x,y,z \in {E} $ and for $\mu,\nu \in 
\mathbb{C}$.
\end{itemize}
 An inner-product $\m B$-module ${E}$
which is complete with respect to the norm  $$\| x\| :=\|\langle
x,x\rangle\|^{1/2} ~\mbox{for}~ x \in {E}$$ is called a
{\rm Hilbert $\m B$-module} or {\rm Hilbert $C^{*}$-module
over $\m B$}. It is said to be {\rm full} if the closure of the
linear span of $\{\langle x,y\rangle:x,y\in{E}\}$ equals $\m B$.
\end{definition}
Hilbert $C^*$-modules are important objects
to study the classification theory of $C^*$-algebras, the
dilation theory of semigroups of completely positive maps, and so on. 
If a completely positive map takes values in any von Neumann algebra, then it gives us a von Neumann module by 
Paschke's GNS construction (cf. \cite{Sk01}). The von Neumann modules were recently utilized  in
\cite{BSu13} to explore Bures distance between two completely
positive maps. Using the following definition of adjointable maps we define von Neumann modules: Let $E$ and $ F$ be (pre-)Hilbert $\m
A$-modules, where $\m A$ is a (pre-)$C^*$-algebra. A map $S:E\to F$
is called {\it adjointable} if there exists a map $S': F\to E$ such
that
 \[
 \langle S (x),y\rangle =\langle x,S'(y) \rangle~\mbox{for all}~x\in E, y\in  F.
 \]
$S'$ is unique for each $S$, henceforth we denote it by $S^{*}$. We denote
the set of all adjointable maps from $E$ to $ F$ by $\ma B^a (E,F)$
and we use $\ma B^a (E)$ for $\ma B^a (E,E)$. Symbols $\ma B(E, F)$
and $\ma B^r (E,F)$ represent the set of all bounded linear maps from $E$ to $ F$ and the set of all bounded 
right linear maps from $E$ to $F$, respectively.  

\begin{definition}(cf. \cite{Sk00})
Let $\m B$ be a von Neumann algebra acting on a Hilbert space $\m
H$, i.e., strongly closed $C^*$-subalgebra of $\ma B(\m H)$
containing the identity operator. Let $E$ be a (pre-)Hilbert $\m
B$-module. The Hilbert space $E\bigodot \m H$ is the interior tensor
product of $E$ and $\m H$. For each $x\in E$ we get a bounded linear
map from $\m H$ to $E\bigodot \m H$ defined as
$$L_x (h):=x\odot h~\mbox{for all}~ h\in \m H.$$
Note that $L^*_{x_1} L_{x_2} =\langle x_1,x_2\rangle~\mbox{for}~
x_1,x_2\in E.$ So we identify each $x\in E$ with $L_x$ and consider
$E$ as a concrete submodule of $\ma B(\m H,E\bigodot \m H)$. The module $E$ is called a {\rm von Neumann $\m B$-module} or a {\rm von Neumann
module over $\m B$} if $E$ is strongly closed in $\ma B(\m
H,E\bigodot \m H)$. Let $\m A$ be a unital (pre-)$C^*$-algebra. A
von Neumann $\m B$-module $E$ is called a {\rm von Neumann $\m
A$-$\m B$ module} if there exists an adjointable left action of $\m
A$ on $E$.
\end{definition}

An alternate approach to the theory of von Neumann modules 
is introduced recently in  \cite{BMSS12} and an analogue of the Stinespring's theorem for
von Neumann bimodules is discussed. The comparison of results coming from these two approach is provided by \cite{Ske12}.

Let $G$ be a locally compact group and let $M(\m A)$ denote the multiplier algebra of any
$C^*$-algebra $\m A$. An {\it action of $G$ on $\m A$} is defined as
a group homomorphism $\alpha:G\to Aut(\m A)$. If $t\mapsto
\alpha_{t}(a)$ is continuous for all $a\in\m A$, then we call
$(G,\alpha,\m A)$ a {\it $C^*$-dynamical system}. 
\begin{definition}\label{def3}(cf. \cite{Kap93})
Let $\m A$, $\m B$ be unital (pre-)$C^*$-algebras and $G$ be a locally compact group. Let $(G,\alpha,\m A)$ be a $C^*$-dynamical system and 
$u:G\to \m U\m B$
be a unitary representation where $\m U \m
B$ is the group of all unitary elements of $\m B$. A
completely positive map $\tau:\m A\to\m B$ is called {\rm
$u$-covariant} with respect to $(G,\alpha,\m A)$ if
\[
\tau(\alpha_{t}(a))=u_{t}\tau(a)u^{*}_t~\mbox{for all}~a\in\m
A~\mbox{and}~t\in G.
\]

\end{definition}

The existence of covariant completely
positive liftings (cf. \cite{CE76}) and a covariant version of the Stinespring's
theorem for operator-valued $u$-covariant completely positi-
ve maps were obtained by Paulsen in \cite{Pau82}, and they were used to provide three
groups out of equivalence classes of covariant extensions. Later
Kaplan (cf. \cite{Kap93}) extended this covariant version and as an
application analyzed the completely positive lifting problem for
homomorphisms of the reduced group $C^*$-algebras. 



A map ${T}$ from a (pre-)Hilbert $\m A$-module ${E}$ to a (pre-)Hilbert
$\m B$-module ${F}$ is called {\it $\tau$-map} (cf. \cite{SSu14}) if
$$\langle T(x),T(y)\rangle=\tau(\langle x,y\rangle)~\mbox{for all} ~x,y\in{E}.$$  Recently a Stinespring type theorem for $\tau$-maps was
obtained by Bhat, Ramesh and Sumesh (cf. \cite{BRS12}) for any operator valued completely positive map $\tau$ defined on a unital $C^*$-algebra.
There are two covariant versions of this Stinespring type theorem see Theorem 3.4 of \cite{Jo11} and Theorem 3.2 of \cite{HJ11}. In Section \ref{sec1.2},
we give a Stinespring type theorem for
$\tau$-maps, when $\m B$ is any von Neumann algebra and $F$ is any von Neumann $\m B$-module. 

In \cite{DH14} the notion of 
$\mathfrak K$-families is introduced, which is a generalization of the $\tau$-maps, and 
several results are derived for covariant $\mathfrak K$-families. In \cite{SSu14} different characterizations of the $\tau$-maps were obtained and as an application the dilation theory of semigroups of the completely positive maps was discussed. Extending some of these results for $\mathfrak K$-families, application to the dilation theory of semigroups of 
completely positive definite kernels is explored in \cite{DH14}.  

In this article we get a covariant version 
of our Stinespring type theorem
which requires the following notions: Let $\m A$ and $\m B$ be
$C^*$-algebras, $E$ be a Hilbert $\m A$-module, and let $ F$ be a
Hilbert $\m B$-module. A map $\Psi:E\to F$  is said to be a {\it
morphism of Hilbert $C^*$-modules} if there exists a $C^*$-algebra
homomorphism $\psi:\m A\to\m B$ such that
$$\langle \Psi(x),\Psi(y)\rangle=\psi(\langle
x,y\rangle)~\mbox{for all}~ x,y\in E.$$ If $E$ is full, then $\psi$
is unique for $\Psi$. A bijective map $\Psi:E\to F$ is called an
{\it isomorphism of Hilbert $C^*$-modules} if $\Psi$ and $\Psi^{-1}$
are morphisms of Hilbert $C^*$-modules. We denote the group of all
isomorphisms of Hilbert $C^*$-modules from $E$ to itself by
$Aut(E)$.

\begin{definition}\label{def1.8}
Let $G$ be a locally compact group and let $\m A$ be a
$C^*$-algebra. Let $E$ be a full Hilbert $\m A$-module. A group
homomorphism $t\mapsto\eta_{t}$ from $G$ to $Aut({E})$ is called a
{\em continuous  action of $G$ on ${E}$} if $t\mapsto \eta_{t}(x)$
from $G$ to ${E}$ is continuous for each $x\in E$. In this case we
call the triple $(G,\eta,E)$ a {\rm dynamical system on the Hilbert
$\m A$-module $E$}. Any $C^*$-dynamical system $(G,\alpha,\m A)$ can
be regarded as a dynamical system on the Hilbert $\m A$-module $\m
A$.
\end{definition}


Let $E$ be a full Hilbert $C^*$-module over a unital $C^*$-algebra $\m A$. Let $ F$ be a von Neumann $\m
C$-$\m B$ module, where $\m C$ is a unital $C^*$-algebra and $\m B$ is a von Neumann algebra. We define covariant $\tau$-maps with respect to
$(G,\eta,E)$ in Section \ref{sec1.2}, and develop a covariant version of our Stinespring type theorem. If $(G,\eta,E)$ is a dynamical system on $E$, 
then there exists a crossed product Hilbert $C^*$-module $E\times_{\eta} G$ (cf. \cite{EKQR00}). 
In Section \ref{sec1.3}, we prove that any
 $\tau$-map from $E$ to $F$ which is $(u',u)$-covariant with respect to the
dynamical system $(G,\eta,E)$ extends to a $(u',u)$-covariant
$\widetilde{\tau}$-map from $E\times_{\eta} G$ to $F$, where $\tau$ and $\widetilde{\tau}$ are
completely positive.  As an application we describe how covariant $\tau$-maps
on $(G,\eta,E)$ and covariant $\widetilde{\tau}$-maps on
$E\times_{\eta} G$  are related, where $\tau$ and $\widetilde{\tau}$ are
completely positive maps. The approach in this article is similar to \cite{BRS12} and \cite{Jo11}. 
\section{A Stinespring type theorem and its covariant version}\label{sec1.2}
\begin{definition}\label{def1.1}
Let $\m A$ and $\m B$ be (pre-)$C^*$-algebras. Let $E$ be a Hilbert $\m A$-module and let $ F$, $ F'$ be inner product $\m B$-modules.
A map $\Psi:E\to \ma B^r( F, F')$ is called {\rm
quasi-representation} if there exists a
$*$-homomorphism $\pi:\m A\to \ma B^a ( F)$ satisfying
$$ \langle\Psi(y) f_1,\Psi(x) f_2\rangle=\langle\pi(\langle
x,y\rangle)f_1,f_2\rangle~\mbox{for all}~x,y\in
E~\mbox{and}~f_1,f_2\in F.$$ In this case we say that $\Psi$ is a quasi-representation of $E$ on $ F$ and $ F'$, and $\pi$ is associated
to $\Psi$. 
\end{definition}

It is clear that Definition \ref{def1.1} generalizes the notion of representations of Hilbert $C^*$-modules on Hilbert spaces (cf. p.804 of \cite{Jo11}). 
The following theorem provides a decomposition of $\tau$-maps in terms of quasi-representations. We use the symbol sot-$\lim$ for the limit
with respect to the strong
operator topology. Notation $[S]$ will be used for the norm closure of
the linear span of any set $S.$

\begin{theorem}\label{prop1.3}
Let $\m A$ be a unital $C^*$-algebra and let $\m B$ be a von Neumann
algebra acting on a Hilbert space $\m H$. Let $E$ be a Hilbert $\m
A$-module, $ E'$ be a von Neumann $\m B$-module and let $\tau:\m
A\to\m B$ be a completely positive map. If $T:E\to E'$ is a
$\tau$-map, then there  exist

\begin{itemize}
\item [(i)]
\begin{enumerate}
\item [(a)] a von Neumann $\m B$-module $F$ and a representation $\pi$ of $\m A$ to $\ma B^a (F)$,
\item [(b)] a map $V\in \ma B^a (\m B,F)$ such that
$\tau(a)b=V^{*}\pi(a)Vb~\mbox{for all}~a\in\m A~\mbox{and}~b\in \m
B,~$
\end{enumerate}
\item [(ii)]
\begin{enumerate}
\item [(a)] a von Neumann $\m B$-module $ F'$ and a quasi-representation $\Psi:E\to \ma B^a (F, F')$ such that $\pi$ is associated to $\Psi$,
\item [(b)] a coisometry $S$ from $ E'$ onto $ F'$ satisfying
 $$T(x)b=S^{*}\Psi(x)Vb~\mbox{for all}~x\in E~\mbox{and}~b\in \m
B.$$
\end{enumerate}
\end{itemize}
\end{theorem}
\begin{proof}
 Let $\langle~,~\rangle$ be a $\m B$-valued positive definite
semi-inner product on $\m A\bigotimes_{alg} \m B$ defined by
$$\langle a\otimes b, c\otimes d\rangle:=b^*\tau(a^* c)d~\mbox{for}~a,c\in \m A~\mbox{and}~b,d\in\m B.$$
Using Cauchy-Schwarz inequality we deduce that $K=\{x\in {\m
A\bigotimes_{alg} \m B}:\langle x, x\rangle=0\}$ is a submodule of
$\m A\bigotimes_{alg} \m B$. Therefore $\langle~,~\rangle$ extends
naturally on the quotient module $\left({\m A\bigotimes_{alg} \m
B}\right)/ K$ as a $\m B$-valued inner product. We get a Stinespring triple
$(\pi_0, V, F_0)$ associated to $\tau$, construction is similar to
Proposition 1 of \cite{Kap93}, where $F_0$ is the completion of the
inner-product $\m B$-module $\left({\m A\bigotimes_{alg} \m
B}\right)/ K$, $\pi_0:\m A\to \ma B^a (F_0)$ is a $*$-homomorphism
defined by
$$\pi_0(a')(a\otimes b+ K):= a' a\otimes b+ K~\mbox{for all}~a,a'\in\m A~\mbox{and}~b\in \m B,$$
and a mapping $V\in \ma B^a(\m B, F_0)$ is defined by
\[
V(b)=1\otimes b+ K~\mbox{for all}~b\in \m B.\] Indeed, $[\pi_0(\m
A)V \m B]=F_0$. Let $F$ be the strong operator topology closure of
$F_0$ in $\ma B(\m H,F_0\bigodot\m  H)$. Without loss of generality
we can consider $V\in \ma B^a(\m B, F)$. Adjointable left action of
$\m A$ on $F_0$ extends to an adjointable left action of $\m A$ on
$F$ as follows:
\[
 \pi(a)(f):=\mbox{sot-}\displaystyle\lim_{\alpha} \pi_0(f^0_{\alpha})~\mbox{where $a\in \m A$,
 $f$=sot-$\displaystyle\lim_{\alpha} f^0_{\alpha}\in F$ with $f^0_{\alpha}\in F_0$.}
\]
For all $a\in \m A$; $f$=sot-$\displaystyle\lim_{\alpha} f^0_{\alpha}$,
$g$=sot-$\displaystyle\lim_{\beta} g^0_{\beta}\in F$ with
$f^0_{\alpha},g^0_{\beta}\in F_0$ we have

\begin{align*}
 \langle \pi(a) f,g\rangle &=\mbox{sot-}\displaystyle\lim_{\beta}\langle \pi(a) f,g^0_{\beta}\rangle
 =\mbox{sot-}\displaystyle\lim_{\beta}(\mbox{sot-}\displaystyle\lim_{\alpha}\langle g^0_{\beta},\pi_0(a) f^0_{\alpha}\rangle)^*
 \\ &=\mbox{sot-}\displaystyle\lim_{\beta}(\mbox{sot-}\displaystyle\lim_{\alpha}\langle \pi_0(a)^*g^0_{\beta}, f^0_{\alpha}\rangle)^*
 = \langle  f,\pi(a^*)g\rangle.
\end{align*}
The triple $(\pi, V, F)$ satisfies all the conditions of the statement
(i).

Let $ F''$ be the Hilbert $\m B$-module $[T(E)\m B]$. For $x\in E$,
define $\Psi_0(x):F_0\to F''$ by
\begin{align*}\Psi_0(x)(\displaystyle\sum_{j=1}^{n}\pi_0(a_j)Vb_j):=\displaystyle\sum_{j=1}^{n}T(xa_j)b_j~\mbox{
for all}~a_j\in \m A, b_j\in\m B.
\end{align*}
It follows that
\begin{align*}&\langle\Psi_0(y)
(\displaystyle\sum_{j=1}^n\pi_0(a_j)Vb_j),\Psi_0(x)(\displaystyle\sum_{i=1}^m\pi_0(a'_i)Vb'_i)\rangle
 = \displaystyle\sum_{j=1}^n \displaystyle\sum_{i=1}^m
b_j^{*}\langle T(ya_j),T(xa'_i)\rangle b'_i\\
&=\displaystyle\sum_{i=1}^m \displaystyle\sum_{j=1}^n
b_j^{*}\tau(\langle ya_j,xa'_i\rangle) b'_i =
\displaystyle\sum_{j=1}^n \displaystyle\sum_{i=1}^m\langle
\pi_0(a'_i)^*\pi_0(\langle x,y\rangle)\pi_0(a_j)Vb_j,Vb'_i\rangle
\\& = \langle \pi_0(\langle
x,y\rangle)(\displaystyle\sum_{j=1}^n\pi_0(a_j)Vb_j),
\displaystyle\sum_{i=1}^m\pi_0(a'_i)Vb'_i\rangle
\end{align*}
for all $x,y\in E, a'_i,a_j\in\m A, b'_i,b_j\in\m B~\mbox{where}~
1\leq j\leq n~\mbox{and}~1\leq i\leq m$. This computation proves
that $\Psi_0(x) \in \ma B^r(F_0, F'')$ for each $x\in E$ and also
that $\Psi_0:E\to \ma B^r(F_0, F'')$ is a quasi-representation. We
denote by $ F'$ the strong operator topology closure of $F''$ in
$\ma B(\m H, E'\bigodot \m H)$. Let $x\in E$, and let $\Psi(x):F\to
F'$ be a mapping defined by
\[
\Psi(x)(f):=\mbox{sot-}\lim_{\alpha} \Psi_0 (x)
f^0_{\alpha}~\mbox{where $f$=sot-$\displaystyle\lim_{\alpha}
f^0_{\alpha}\in F$ for $f^0_{\alpha}\in F_0$.}~
\]
 For all
$f$=sot-$\displaystyle\lim_{\alpha} f^0_{\alpha}\in F$ with $f^0_{\alpha}
\in F_0$ and for all $x,y\in E$ we have
\begin{align*}
\langle \Psi(x) f,\Psi(y) f\rangle=
\mbox{sot-}\lim_{\alpha}\{\mbox{sot-}\lim_{\beta}\langle
\Psi_0(y)f^0_{\alpha}, \Psi_0(x)f^0_{\beta}\rangle\}^*=\langle
f,\pi(\langle x,y\rangle)f\rangle.
\end{align*}
Since $F$ is a von Neumann $\m B$-module, this proves that
$\Psi:E\to \ma B^a (F, F')$ is a quasi-representation. Since, $ F'$
is a von Neumann $\m B$-submodule of $ E'$, there exists an
orthogonal projection
 from $ E'$ onto $ F'$ (cf. Theorem 5.2 of \cite{Sk00}) which we denote by $S$. Eventually
\begin{eqnarray*}
S^*\Psi(x) Vb=\Psi(x)Vb=\Psi(x)(\pi(1)Vb)=T(x)b~\mbox{for all}~x\in
E, b\in\m B.
\end{eqnarray*}

\end{proof}

Let $E$ be a (pre-)Hilbert $\m A$-module, where $\m A$ is a
(pre-)$C^*$-algebra $\m A$. A map $u\in \ma B^a (E)$ is said to be
{\it unitary} if $u^{*} u=u u^{*}=1_{E}$ where $1_{E}$ is the
identity operator on $E$. We denote the set of all unitaries in $\ma
B^a (E)$ by $\m U \ma B^a (E)$.

\begin{definition}\label{def1.4}
Let $\m B$ be a (pre-)$C^*$-algebra,
$(G,\alpha,\m A)$ be a $C^*$-dynamical system of a locally compact
group $G$, and let $ F$ be a (pre-)Hilbert $\m B$-module. A representation
$\pi:\m A\to \ma B^a ( F)$ is called {\rm $v$-covariant} with respect to $(G,\alpha,\m A)$ and with respect to a unitary representation $v:G\to
\m U \ma B^a ( F)$ if
\[
\pi(\alpha_{t}(a))=v_{t}\pi(a)v^{*}_{t}~\mbox{for all}~a\in\m A,t\in
G.
\]
In this case we write $(\pi,v)$ is a covariant representation of
$(G,\alpha, \m A)$.
\end{definition}

Let $E$ be a full Hilbert $\m A$-module and let $G$ be a locally compact group. If $(G,\eta,E)$ is a
dynamical system on $E$, then there exists a
unique $C^*$-dynamical system $(G,\alpha^{\eta},\m A)$ (cf. p.806 of
\cite{Jo11}) such that
$$\alpha^{\eta}_t(\langle x,y \rangle)=\langle {\eta}_{t}(x),
{\eta}_{t}(y)\rangle~\mbox{for all}~x,y\in E~\mbox{and}~t\in G.$$ We
denote by $(G,\alpha^{\eta},\m A)$ the $C^*$-dynamical system coming
from the dynamical system $(G,\eta,E)$. For all $x\in E$ and $a\in \m A$ we infer that $\eta_t
(xa)=\eta_t(x) \alpha^{\eta}_t (a)$, for
\begin{align*}
 \|\eta_t (xa)-\eta_t(x)\alpha^{\eta}_t (a)\|^2 =&\| \langle \eta_t (xa), \eta_t (xa)\rangle
 - \langle\eta_t (xa), \eta_t (x)\alpha^{\eta}_t (a)\rangle\\ &-\langle \eta_t (x)\alpha^{\eta}_t (a),\eta_t (xa)\rangle
 +\langle\eta_t (x)\alpha^{\eta}_t (a),\eta_t (x)\alpha^{\eta}_t (a)\rangle\|\\ =&\| \alpha^{\eta}_t(\langle xa, xa\rangle)
 - \langle\eta_t (xa), \eta_t (x)\rangle\alpha^{\eta}_t (a)\\ &-\alpha^{\eta}_t (a^*)\langle \eta_t (x),\eta_t (xa)\rangle 
 +\alpha^{\eta}_t (a^*)\langle\eta_t (x),\eta_t (x)\rangle\alpha^{\eta}_t (a)\| =& 0.
\end{align*}

\begin{definition}
Let $\m B$ and $\m C$ be unital (pre-)$C^*$-algebras. A {\rm
(pre-)$C^*$-correspondence from $\m C$ to $\m B$} is defined as a
(pre-)Hilbert $\m B$-module $ F$ together with a $*$-homomorphism $\pi':\m
C\to \ma B^a ( F)$. The adjointable left action of $\m C$ on $ F$ induced by
$\pi'$ is defined as
\[
cy:=\pi'(c)y~\mbox{for all}~c\in \m C,y\in F.
\]
\end{definition}

In the remaining part of this section a covariant version of Theorem \ref{prop1.3} is derived, which finds applications in the next section.
For that we first define covariant $\tau$-maps using the notion of (pre-)$C^*$-correspondence. 
Every von Neumann $\m B$-module $E$ can be considered as a (pre-)$C^*$-correspondence from $\ma B^a (E)$ to
$\m B$. 

\begin{definition} (cf. \cite{Jo11})
Let $\m A$ be a unital $C^*$-algebra and let $\m B$, $\m C$ be unital
(pre-)$C^*$-algebras. Let $E$ be a Hilbert $\m A$-module and let $
F$ be a (pre-)$C^*$-correspondence from $\m C$ to $\m B$. Let $u:G\to \m U \m B$ and $u':G\to \m U
\m C$ be unitary representations on a locally compact group $G$. A $\tau$-map, $T:E\to  F$, is called {\rm
$(u',u)$-covariant} with respect to the dynamical system $(G,\eta,E)$ if
\[
T(\eta_{t}(x))=u'_tT(x)u^{*}_t~\mbox{for all}~x\in E~\mbox{and}~t\in
G.
\]
\end{definition}

If $E$ is full and $T:E\to  F$ is a $\tau$-map which is $(u',u)$-covariant with respect to $(G,\eta,E)$, then the map $\tau$ is $u$-covariant with respect to
the induced $C^*$-dynamical system $(G,\alpha^{\eta},\m A)$, because
\begin{align*}
\tau(\alpha^{\eta}_t(\langle x,y \rangle))&=\tau(\langle
{\eta}_{t}(x), {\eta}_{t}(y)\rangle)=\langle T({\eta}_{t}(x)),
T({\eta}_{t}(y))\rangle=\langle
u'_tT(x)u^{*}_t,u'_tT(y)u^{*}_t\rangle
\\&=\langle T(x)u^{*}_t,T(y)u^{*}_t\rangle = u_t \langle
T(x),T(y)\rangle u^{*}_t=u_t\tau(\langle x,y\rangle)u^{*}_t
\end{align*}
for all $x,y\in E$ and $t\in G$.

\begin{definition} Let $(G,\eta,E)$ be a dynamical system on a Hilbert
$\m A$-module $E$, where $\m A$ is a $C^*$-algebra. Let $ F$ and $ F'$ be Hilbert $\m B$-modules over
a (pre-)$C^*$-algebra $\m B$. $w:G\to \m U\ma B^a ( F')$ and $v:G\to
\m U\ma B^a ( F)$ are unitary representations on a locally compact group $G$. A
quasi-representation of $E$ on $ F$ and $ F'$ is called {\rm
$(w,v)$-covariant with respect to $(G,\eta,E)$} if
\[
\Psi(\eta_{t}(x))=w_t\Psi(x)v^{*}_t~\mbox{for all}~x\in
E~\mbox{and}~t\in G.
\]
In this case we say that $(\Psi,v,w, F, F')$ is a covariant
quasi-representation of $(G,\eta,E)$. Any $v$-covariant
representation of a $C^*$-dynamical system $(G,\alpha,\m A)$ can be
regarded as a $(v,v)$-covariant representation of a dynamical system
on the Hilbert $\m A$-module $\m A$.
\end{definition}

Let $\m A$ be a $C^*$-algebra and let $G$ be a locally compact
group. Let $E$ be a full Hilbert $\m A$-module, and let $ F$ and $
F'$ be Hilbert $\m B$-modules over a (pre-)$C^*$-algebra $\m B$. If
$(\Psi,v,w, F, F')$ is a covariant quasi-representation  with respect to $(G,\eta,E)$, then the representation
of $\m A$ associated to $\Psi$ is $v$-covariant with respect to $(G,\alpha^{\eta},\m A)$. Moreover, if $\pi$ is the representation
associated to $\Psi$, then 
\begin{align*}
\langle\pi(\alpha^{\eta}_t(\langle x,y \rangle))f,f'\rangle
&=\langle\pi(\langle {\eta}_{t}(x),
{\eta}_{t}(y)\rangle)f,f'\rangle=\langle
\Psi({\eta}_{t}(y))f,\Psi({\eta}_{t}(x)) f'\rangle\\&= \langle
w_t\Psi(y)v^{*}_t f, w_t\Psi(x)v^{*}_t f'\rangle=\langle
v_t\pi(\langle x,y\rangle)v^{*}_t f,f'\rangle
\end{align*}
for all $x,y\in E$, $t\in G$ and $f,f'\in F$.

\begin{theorem}\label{prop1.6}
Let $\m A$, $\m C$ be unital $C^*$-algebras and let $\m B$ be a von
Neumann algebra acting on $\m H$. Let $u:G\to \m U \m B$,
$u':G\to \m U \m C$ be unitary representations of a locally compact
group $G$. Let $E$ be a full Hilbert $\m A$-module and $E'$ be a von
Neumann $\m C$-$\m B$ module. If $T:E\to E'$ is a $\tau$-map which
is $(u',u)$-covariant with respect to $(G,\eta,E)$ and if $\tau:\m
A\to \m B$ is completely positive, then there exists

\begin{itemize}
\item [(i)]
\begin{enumerate}
\item [(a)] a von Neumann $\m B$-module $F$ with a covariant representation $(\pi,
v)$ of $(G,\alpha^{\eta},\m A)$ to $\ma B^a (F)$,
\item [(b)] a map $V\in \ma B^a (\m B,F)$ such that
\begin{enumerate}
\item [(1)] $\tau(a)b=V^{*}\pi(a)Vb~\mbox{for all}~a\in\m A,~b\in \m
B,~$
\item [(2)] $v_{t}Vb=Vu_{t}b$ for all $t\in G$, $b\in \m B$,
\end{enumerate}
\end{enumerate}
\item [(ii)]
\begin{enumerate}
\item [(a)] a von Neumann $\m B$-module $ F'$ and a covariant quasi-representation \newline $(\Psi,v,
w,F, F')$ of $(G,\eta,E)$ such that $\pi$ is associated to $\Psi$,
\item [(b)] a coisometry $S$ from $ E'$ onto $ F'$ such that
\begin{enumerate}
\item [(1)] $T(x)b=S^{*}\Psi(x)Vb~\mbox{for all}~x\in E,~b\in \m
B,$
\item [(2)] $w_{t}Sy=Su'_{t}y$ for all $t\in G$, $y\in  E'$.
\end{enumerate}
\end{enumerate}
\end{itemize}
\end{theorem}
\begin{proof}
By part (i) of Theorem \ref{prop1.3} we obtain the triple $(\pi,V, F)$
associated to $\tau$. Here $F$ is a von Neumann $\m B$-module,
 $V\in \ma B^a (\m B,F)$, and $\pi$ is a representation of $\m A$ to $\ma B^a (F)$ such that
$$\tau(a)b=V^{*}\pi(a)Vb~\mbox{for all}~a\in\m A,~b\in \m
B.~$$ Recall the proof, using the submodule $K$ we have constructed
the triple $(\pi_0, V, F_0)$ with $[\pi_0(\m A)V\m B]=F_0$. Define
$v^0: G\to \ma B^a (F_0)$ (cf. Theorem 3.1, \cite{He99}) by
$$v^0_t(a\otimes b+ K):=\alpha_t(a)\otimes u_{t}(b) + K~\mbox{for all $a\in\m A$, $b\in \m B$ and $t\in G.$}~$$
 Since $\tau$ is $u$-covariant with respect to
$(G,\alpha^{\eta},\m A)$, for $a,a'\in\m A$, $b,b'\in \m B$
and $t\in G$ it follows that
 \begin{align*}
  \langle v^0_t a\otimes b+ K, v^0_t a'\otimes b'+ K\rangle 
  &=\langle \alpha_t (a)\otimes u_t b,\alpha_t (a')\otimes u_t b'\rangle=(u_t b)^* \tau(\alpha_t(a^*a'))u_t b'
  \\ &=b^*\tau(a^*a')b'=\langle a\otimes b+ K, a'\otimes b'+ K\rangle.
 \end{align*}
This map $v^0_t$ extends as a unitary on $F_0$ for each $t\in G$ and
further we get a group homomorphism $v^0:G\to \m U \ma B^a (\m
F_0)$. The continuity of $t\mapsto\alpha^{\eta}_{t}(b)$ for each
$b\in\m B$, the continuity of $u$ and the fact that $v^0_t$ is a
unitary for each $t\in G$ together implies the continuity of $v^0$.
Thus $v^0:G\to \m U \ma B^a (\m F_0)$ becomes a unitary
representation. For each $t\in G$ define $v_t :F\to F$ by
\[
 v_t (\mbox{sot-}\lim_{\alpha} f^0_{\alpha}):=\mbox{sot-}\lim_{\alpha} v^0_t (f^0_{\alpha})~\mbox{where $f$
 =sot-$\displaystyle\lim_{\alpha} f^0_{\alpha}\in F$ for
$f^0_{\alpha}\in F_0$.}~
\]
It is clear that $v:G\to \ma B^a (F)$  is a unitary representation
of $G$ on $F$ and moreover it satisfies the condition (i)(b)(2) of
the statement.

Notation $F''$ will be used for $[T(E)\m B]$ which is a Hilbert $\m
B$-module. Let $ F'$ be the strong operator topology closure of $
F''$ in $\ma B(\m H, E'\bigodot \m H)$. For each $x\in E$, define
$\Psi_0(x):F_0\to F''$ by
\[\Psi_0(x)(\displaystyle\sum_{j=1}^{n}\pi(a_j)Vb_j):=\displaystyle\sum_{j=1}^{n}T(xa_j)b_j~\mbox{
for all}~a_j\in \m A, b_j\in\m B
\]
and define $\Psi(x):F\to  F'$ by
\[
\Psi(x)(f):=\mbox{sot-}\lim_{\alpha} \Psi_0 (x)
f^0_{\alpha}~\mbox{where $f$=sot-$\displaystyle\lim_{\alpha}
f^0_{\alpha}\in F$ for $f^0_{\alpha}\in F_0$.}~
\]
$\Psi_0:E\to \ma B^r(F_0, F'')$ and $\Psi:E\to \ma B^a (F, F')$ are
quasi-representations (see part (ii) of Theorem \ref{prop1.3}).
Indeed, there exists an orthogonal projection $S$ from $ E'$ onto $
F'$ such that
$$T(x)b=S^{*}\Psi(x)Vb~\mbox{for all}~x\in E~\mbox{and}~b\in \m
B.$$ Since $T$ is $(u',u)$-covariant, we have
\[u'_t(\displaystyle\sum_{i=1}^{n}T(x_i)b_i)=\displaystyle\sum_{i=1}^{n}T(\eta_{t}(x_i))u_{t}b_i ~\mbox{for all}~t\in G,x_i\in E,b_i\in \m B,
i=1,2,\ldots,n.\] From this computation it is clear that $ F''$ is
invariant under $u'$. For each $t\in G$ define $w^0_t:=u'_t|_{
F''}$, the restriction of $u'_t$ to $ F''$. In fact, $t\mapsto
w^0_t$ is a unitary representation of $G$ on $ F''$. Further
\begin{align*}
&\Psi_0(\eta_{t}(x))(\displaystyle\sum_{i=1}^{n}\pi_0(a_i)Vb_i)=
\displaystyle\sum_{i=1}^{n}T(\eta_{t}(x)\alpha^{\eta}_{t}\alpha^{\eta}_{t^{-1}}(a_i))b_i
=
\displaystyle\sum_{i=1}^{n}T(\eta_{t}(x\alpha^{\eta}_{t^{-1}}(a_i)))b_i
\\ &=\displaystyle\sum_{i=1}^{n}u^{'}_tT(x\alpha^{\eta}_{t^{-1}}(a_i))u_{t^{-1}}b_i
=w^0_{t}\Psi_0(x)(\displaystyle\sum_{i=1}^{n}\pi_0(\alpha^{\eta}_{t^{-1}}(a_i))Vu_{t^{-1}}b_i)\\
&=
w^0_{t}\Psi_0(x)v_{t^{-1}}(\displaystyle\sum_{i=1}^{n}\pi_0(a_i)Vb_i)
\end{align*}
for all $a_1,a_2,\ldots,a_n\in \m A$, $b_1,b_2,\ldots,b_n\in \m B,
x\in E,t\in G$. Therefore $(\Psi_0,v^0, w^0,F_0, F'')$ is a
covariant quasi-representation of $(G,\eta,E)$ and $\pi_0$ is
associated to $\Psi_0$. For each $t\in G$ define $w_t: F'\to  F'$ by
\[
 w_t (\mbox{sot-}\displaystyle\lim_{\alpha} f''_{\alpha}):=\mbox{sot-}\displaystyle\lim_{\alpha} u'_t f''_{\alpha}
 ~\mbox{where all $f''_{\alpha}\in  F''$.}~
\]
 It is evident that the map $t\mapsto w_t$ is a unitary
representation of $G$ on $ F'$. $S$ is the orthogonal projection of
$ E'$ onto $ F'$ so we obtain  $w_tS=Su'_t$ on $ F$ for all $t\in
G$. Finally
\begin{align*}
&\Psi(\eta_{t}(x))f=\mbox{sot-}\lim_{\alpha} \Psi_0 (\eta_{t}(x))
f^0_{\alpha} =\mbox{sot-}\lim_{\alpha} w^0_{t}\Psi_0(x)v^0_{t^{-1}}
f^0_{\alpha}=w_{t}\Psi(x)v_{t^{-1}}f
\end{align*}
for all $x\in E$, $t\in G$ and $f$=sot-$\displaystyle\lim_{\alpha}
f^0_{\alpha}\in F$ for $f^0_{\alpha}\in F_0$. Whence $(\Psi,v, w,F, F')$
is a covariant quasi-representation of $(G,\eta,E)$ and observe that $\pi$ is associated to $\Psi$.
\end{proof}

\section{$\tau$-maps from the crossed product of Hilbert $C^*$-modules}\label{sec1.3}

Let $(G,\eta,E)$ be a dynamical system on $E$, which is a full Hilbert $C^*$-module over $\m A$, where $G$ is a locally
compact group. The {\it crossed product} Hilbert $C^*$-module $E\times_{\eta}G$
(cf. Proposition 3.5, \cite{EKQR00}) is the completion of an inner-product $\m
A\times_{\alpha^{\eta}}G$-module $C_c(G,E)$ such that the module action
and the $\m A\times_{\alpha^{\eta}}G$-valued inner product are given
by
\begin{align*}\label{eqn1} lg(s)&=\int_{G}
l(t)\alpha^{\eta}_{t}(g(t^{-1}s))dt,\\
\langle l,m\rangle_{\m A\times_{\alpha^{\eta}}G }(s)&=\int_{G}
\alpha^{\eta}_{t^{-1}}(\langle l(t),m(ts)\rangle)dt
\end{align*}
respectively for $s\in G$, $g\in C_c (G,\m A)$ and $l,m\in C_c
(G,E)$. The following lemma shows that any covariant quasi-representation
$(\Psi_0,v^0,w^0,F_0, F')$ with respect to $(G,\eta,E)$ provides a
quasi-representation $\Psi_0\times v^0$ of $E\times_{\eta}
G$ on $F_0$ and $F'$ satisfying
\[
(\Psi_0\times v^0)(l)=\int_{G} \Psi_0(l(t))v^0_t dt~\mbox{for all $l\in C_c(G,E)$.}
\]
Moreover, it says that if $\pi_0$ is associated to $\Psi_0$, then the
integrated form of the covariant representation $(\pi_0,v^0,F_0)$ with respect to
$(G,\alpha^{\eta},\m A)$ is associated to $\Psi_0\times v^0$.

\begin{lemma}\label{lem1.3}
Let $(G,\eta,E)$ be a dynamical system on
a full Hilbert $\m A$-module $E$, where $\m A$ is a unital $C^*$-algebra and $G$ is a locally compact group. 
Let $F_0$ and $ F'$ be Hilbert $\m B$-modules, where $\m B$ is a von Neumann
algebra acting on a Hilbert space $\m H$. If $(\Psi_0,v^0,w^0,F_0, F')$ is a covariant
quasi-representation with respect to $(G,\eta,E)$, then $\Psi_0\times v^0$ is a
quasi-representation of $E\times_{\eta} G$ on $F_0$ and $ F'$.
\end{lemma}
\begin{proof}
For $l\in C_c(G,E)$ and $g\in C_c(G,\m A)$, we get
\begin{align*}
(\Psi_0\times v^0)(lg)&=\int_{G}\int_{G}\Psi_0(l(t)\alpha^{\eta}_{t}
(g(t^{-1}s))v^0_{s}ds dt
\\&=\int_{G}\int_{G}\Psi_0(l(t))\pi_0(\alpha^{\eta}_{t}
(g(t^{-1}s))v^0_{s}ds dt\\
&=\int_{G}\int_{G}\Psi_0(l(t))v^0_{t}\pi_0
(g(t^{-1}s)v^{0*}_{t}v^0_{s}ds dt \\
&=(\Psi_0\times v^0)(l)(\pi_0\times v^0)(g).
\end{align*}
For $l,m\in C_c(G,E)$ and $f_0,f'_0\in F_0$ we have
\begin{align*}
\langle(\pi_0\times v^0)(\langle l,m\rangle)f_0,f'_0\rangle
&=\left< \int_{G}\pi_0(\langle l,m\rangle(s))v^{0}_{s}f_0 ds,f'_0
\right>
\\
&= \left< \int_{G}\int_{G}v^{0*}_t\pi_0 (\langle
l(t),m(ts)\rangle)v^{0}_{ts} f_0 dt ds,  f'_0 \right> \\&= \int_{G}\int_{G}
\langle \Psi_0(m(ts))v^0_{ts}f_0,\Psi_0(l(t))v^0_t f'_0\rangle dt ds
\\
&=  \left< \int_{G}\Psi_0(m(s))v^0_{s}f_0 ds,\int_{G}\Psi_0(l(t))v^0_t
f'_0dt\right> 
\end{align*}
\begin{align*}
&= \langle(\Psi_0\times v^0)(m)f_0,(\Psi_0\times
v^0)(l)f'_0\rangle.\qedhere
\end{align*}
\end{proof}

\begin{definition}(cf. \cite{Jo11})
 Let $G$ be a locally compact group with the modular function
$\bigtriangleup$. Let $u:G\to \m U \m B$ and $u':G\to \m U \m C$ be unitary representations of $G$ on unital 
 (pre-)$C^*$-algebras $\m B$ and $\m C$, respectively. Let $ F$ be a
(pre-)$C^*$-correspondence from $\m C$ to $\m B$ and let $(G,\eta,E)$ be a
dynamical system on a Hilbert $\m A$-module $E$, where $\m A$ is a unital $C^*$-algebra. A $\tau$-map, $T:E\times_{\eta} G\to  F$, is
called {\rm $(u',u)$-covariant} if
\begin{enumerate}
 \item [(a)] $T(\eta_t \circ m^{l}_t)=u'_t
 T(m)$ where
 $m^{l}_t(s)=m(t^{-1} s)$ for all $s,t\in G$, $m\in C_c(G,E)$;
 \item [(b)] $T(m^{r}_t)=T(m)u_t$ where
 $m^{r}_t (s)=\bigtriangleup(t)^{-1}m(s t^{-1})$ for all $s,t\in G$, $m\in C_c(G,E)$.
\end{enumerate}

\end{definition}

\begin{proposition}\label{prop3.3}
Let $\m B$ be a von
Neumann algebra acting on a Hilbert space $\m H$, $\m C$ be a unital $C^*$-algebra, and let $ F$ be a von Neumann $\m C$-$\m B$ module.
Let $(G,\eta,E)$ be
a dynamical system on a full Hilbert $\m A$-module $E$, where $\m A$ is a unital $C^*$-algebra and $G$ is a locally compact group. 
Let $u:G\to \m U \m B$,  $u':G\to \m U \m C$
be unitary representations  and let $\tau:\m A\to \m B$
be a completely positive map. If $T:E\to F$ is a $\tau$-map which is $(u',u)$-covariant
 with respect to $(G,\eta,E)$, then there exist a
completely positive map $\widetilde{\tau}:\m A\times_{\alpha^{\eta}}G\to \m B$ and a $(u',u)$-covariant map $\widetilde{T}:E\times_{\eta} G\to
F$ which is  a $\widetilde{\tau}$-map. Indeed, $\widetilde{T}$ satisfies
$$\widetilde{T}(l)=\int_G T(l(s))u_s ds~\mbox{for all}~l\in C_c (G,E).$$
\end{proposition}
\begin{proof}
By Theorem \ref{prop1.6} there exists the Stinespring type construction $(\Psi,\pi,v,w,V,$ $S,F,F')$, associated to $T$, based on the construction 
$(\Psi_0,\pi_0,v^0,T,F_0, F'')$.
Define a map $\widetilde{T}:E\times_{\eta} G\to
 F$ by
\[
\widetilde{T}(l):=S^{*}(\Psi_0\times v^0)(l)V,~\mbox{for all}~l\in
C_c(G,E).
\]
Indeed, for all $l\in C_c (G,E)$ we obtain
\begin{align*}\widetilde{T}(l)&=S^{*}(\Psi_0\times v^0)(l)V=S^{*}\int_{G}\Psi_0(l(s))v^0_s
dsV=\int_{G}S^{*}\Psi_0(l(s))V u_s
ds
\\ &=\int_G T(l(s))u_s ds.
\end{align*}
It is clear that $(\pi_0\times v^0,V,F_0)$ is the Stinespring triple (cf. Theorem \ref{prop1.3})
associated to the completely positive map $\widetilde{\tau}:\m A\times_{\alpha^{\eta}}G\to \m B$
defined by \[\widetilde{\tau}(h):=\int_{G}
\tau(f(t))v^0_{t}dt~\mbox{for all}~ f\in C_c (G,\m A);
b,b'\in \m B.
\]
We have
\begin{eqnarray*}
\langle \widetilde{T}(l),\widetilde{T}(m)\rangle b &=&\langle
S^{*}(\Psi_0\times v^0)(m)V,S^{*}(\Psi_0\times v^0)(l)V\rangle
b
=\widetilde{\tau}(\langle l,m\rangle)b
\end{eqnarray*}
for all $l,m\in E\times_{\eta} G,~b\in\m B.$ Hence $\widetilde{T}$ is a
$\widetilde{\tau}$-map. Further,
\begin{align*}
\widetilde{T}(\eta_t \circ m^{l}_t)&=S^*\int_G
\Psi_0(\eta_t(m(t^{-1} s)))v^0_s ds V=S^*\int_G w^0_t
\Psi_0(m(t^{-1} s))v^0_{t^{-1} s} ds V
\\ &= u'_t
 \widetilde{T}(m);\\
\widetilde{T}(m^{r}_t)&=S^*\int_G \bigtriangleup
(t)^{-1}\Psi_0(m(s t^{-1}))v^0_s ds V= S^*\int_G
\Psi_0(m(g))v^0_g v^0_t dg V\\ &=\widetilde{T}(m)u_t~\mbox{where $t\in G$, $m\in C_c(G,E)$.}~
\end{align*}

\end{proof}
Proposition \ref{prop3.3} gives us a map $T\mapsto \widetilde{T}$ where $T:E\to F$ is a $\tau$-map which is $(u',u)$-covariant
 with respect to $(G,\eta,E)$ and $\widetilde{T}:E\times_{\eta} G\to
F$ is $(u',u)$-covariant $\widetilde{\tau}$-map such that $\tau$ and $\widetilde{\tau}$ are completely positive. This map is actually 
a one-to-one correspondence. To prove this result we need the following terminologies:

We identify $M(\m A)$ with $\ma B^a (\m A)$ (cf. Theorem
2.2 of \cite{La95}), here $\m A$ is considered as a Hilbert $\m
A$-module in the natural way. 
The {\it strict topology} on $\ma B^a ({E})$ is the topology given
by the seminorms $a\mapsto \|ax\| $, $a\mapsto \|a^*y\| $ for each
$x,y\in E$. For each $C^*$-dynamical system $(G,\alpha,\m A)$ we get a non-degenerate
faithful homomorphism $i_{\m A}:\m A\to M(\m A\times_{\alpha} G)$ and an injective strictly continuous homomorphism $i_G : G\to \m
UM(\m A \times_{\alpha} G)$ (cf. Proposition 2.34 of \cite{Wi07})
defined by
\[
 i_{\m A}(a)(f)(s):=af(s)~\mbox{for}~a\in\m A,~s\in G,~f\in C_c (G,\m A);
\]
\[
 i_{G} (r) f(s):=\alpha_r (f(r^{-1} s))~\mbox{for}~r,s\in G,~f\in C_c (G,\m A).
\]
Let $E$ be a Hilbert $C^*$-module over a $C^*$-algebra $\m A$.
 Define the {\it multiplier module}
$M(E):=\ma B^a (\m A,E)$. $M(E)$ is a Hilbert $C^*$-module over
$M(\m A)$ (cf. Proposition 1.2 of \cite{RT03}). For a dynamical
system $(G,\eta,E)$ on $E$ we get a
non-degenerate morphism of modules $i_{E}$ from $E$ to
$M(E\times_{\eta} G)$ (cf. Theorem 3.5 of \cite{Jo12b}) as follows:
For each $x\in E$ define $i_{E}(x):C_c(G,\m A)\to C_c(G,E)$ by
$$ i_{E}(x)(f)(s):=xf(s)~\mbox{for all}~f\in C_c(G,\m A),~s\in G.$$
Note that $i_{E}$ is an $i_{\m A}$-map.

\begin{theorem}
 Let $\m A$, $\m C$ be unital $C^*$-algebras, and let $\m B$ be a von Neumann algebra acting on a Hilbert space $\m
H$. Let $u:G\to \m U \m B$, $u':G\to \m U \m C$ be unitary representations of a locally compact group $G$. If $(G,\eta,E)$ is a dynamical system 
on a full Hilbert $\m A$-module $E$, 
and if $ F$ is a von Neumann $\m C$-$\m B$
module, then there
exists a bijective correspondence $\mf{I}$ from the set of all  $\tau$-maps, $T:E\to F$, which are
$(u',u)$-covariant  with respect to
$(G,\eta,E)$ onto the set of all maps $\widetilde{T}:E\times_{\eta} G\to F$ which are $(u',u)$-covariant
$\widetilde{\tau}$-maps  such that $\tau:\m A\to \m B$ and $\widetilde{\tau}:\m
A\times_{\alpha^{\eta}}G\to \m B$ are completely positive maps.
\end{theorem}
\begin{proof}
Proposition \ref{prop3.3} ensures that the map $\mf{I}$ exists and is well-defined. Let $T:E\times_\eta G\to  F$ be a $(u',u)$-covariant
$\tau$-map, where $\tau:\m
A\times_{\alpha^{\eta}}G\to \m B$ is a completely positive map. Suppose $(\Psi_0,\pi_0,V,F_0, F'')$ and $(\Psi,\pi,V,S,F,
F')$ are the Stinespring type constructions associated to $T$ as in the proof of Theorem
\ref{prop1.3}. Let $\{e_i\}_{i\in \m I}$ be an approximate identity
for $\m A\times_{\alpha^{\eta}} G$. Then there exists
a representation $\overline{\pi_0}:M(\m A\times_{\alpha^{\eta}} G)\to \ma B^a (F_0)$
(cf. Proposition 2.39 of \cite{Wi07}) defined by
\[
 \overline{\pi_0}(a)x:=\displaystyle\lim_i \pi_0(a e_i)x~\mbox{for all $a\in M(\m A\times_{\alpha^{\eta}} G)$ and $x\in F_0$.}~
\]
A mapping $\overline{\Psi_0}:M(E \times_{\eta} G)\to \ma B^r(F_0, F'')$ defined by
\[
\overline{\Psi_0}(h)x:=\lim_i \Psi_0(h e_i)x~\mbox{for all $h\in M(E
\times_{\eta} G)$ and $x\in F_0$,}
\]
is a quasi representation and $\overline{\pi_0}$ is associated
to $\overline{\Psi_0}$. If
$\widetilde{\pi_0}:=\overline{\pi_0}\circ i_{\m A}$, then we further get a quasi-representation
$\widetilde{\Psi_0}:E \to \ma B^r(F_0, F'')$
defined as $\widetilde{\Psi_0}:=\overline{\Psi_0}\circ i_{E}$ such that $\widetilde{\pi_0}$ is associated to $\widetilde{\Psi_0}$. Define maps
$T_0:E\to F$ and $\tau_0:\m A\to \m B$ by
\[
 T_0(x)b:=S^*\widetilde{\Psi_0}(x)Vb~\mbox{for}~b\in \m B,~x\in E~{and}
\]
\[
 \tau_0(a):=V^*\widetilde{\pi_0}(a)V~\mbox{for all}~a\in\m A.
\]
It follows that $\tau_0$ is a completely positive map and $T_0$ is a $\tau_0$-map.

Let $v^0:G\to \m U\ma B^a (F_0)$ be a unitary representation
defined by $v^0:=\overline{\pi_0}\circ i_G$ where $$i_G
(t)(f)(s):=\alpha_t (f(t^{-1} s))~\mbox{for all}~t,s\in G,~f\in C_c (G,\m A).$$
Observe that $\widetilde{\pi_0}:\m A\to \ma B^a (F_0)$ is a
$v^0$-covariant and $\widetilde{\pi_0} \times v^0=\pi_0$ (cf.
Proposition 2.39, \cite{Wi07}). We extend $v^0$ to a unitary representation
$v:G\to \m U\ma B^a (F)$ as in the proof of Theorem \ref{prop1.6}. 
It is easy to verify that $$\alpha^\eta_t\circ \langle
m,m'\rangle^l_t=\langle
m^r_{t^{-1}},m'\rangle~\mbox{for all}~m,m'\in C_c (G,E).$$ Using the fact that $T$ is
$(u',u)$-covariant we get
\[
\tau(\alpha^\eta_t\circ \langle
m,m'\rangle^l_t)=\tau(\langle
m^r_{t^{-1}},m'\rangle)=\langle T(m)
u_{t^{-1}} ,T(m')\rangle=u_t \tau (\langle
m,m'\rangle),
\]
$\mbox{for all
$m,m'\in C_c (G,E)$.}$ Therefore we have
\begin{eqnarray*}
 \langle v_t(\pi_0 (f)V b), Vb'\rangle &=&  \langle v^0_t ((\widetilde{\pi_0} \times v^0)(f)V b, Vb'\rangle
 =\left< \int_G \widetilde{\pi_0} (\alpha^\eta_t (f(s))) v^0_{ts} V b ds,Vb'\right>
 \\ &=&\langle (\widetilde{\pi_0}\times v^0) (\alpha^\eta_t \circ f^l_t)V b, Vb'\rangle
  =\langle\tau(\alpha^\eta_t \circ f^l_t)b, b'\rangle\\ &=& \langle(\pi_0(f)V b),Vu_{t^{-1}}b'\rangle
\end{eqnarray*}
for all $t\in G$, $b,b'\in \m B$ and $f\in C_c(G, \m A)$.
This implies that $v_t V=Vu_t$ for each $t\in G$. For each $t\in G$ define $w^0_t
:[\widetilde{\Psi_0}(E)V\m B]\to [\widetilde{\Psi_0}(E)V\m B]$ by
\[
 w^0_t (\widetilde{\Psi_0}(x)Vb):=\widetilde{\Psi_0}(\eta_t(x))Vu_t b~\mbox{for all $x\in E$, $b\in \m B$.}~
\]
Let $t\in G$, $x,y\in E$ and $b,b'\in \m B$. Then
\begin{align*}
 &\langle \widetilde{\Psi_0}(\eta_t (x)) Vu_t b, \widetilde{\Psi_0}(\eta_t (y)) Vu_t b'\rangle
 \\=&\langle \widetilde{\pi_0}(\langle \eta_t (y),\eta_t(x)\rangle)Vu_t b, Vu_t b'\rangle
 =\langle \widetilde{\pi_0}(\alpha^{\eta}_t(\langle  y,x\rangle))Vu_t b, Vu_t b'\rangle
\\ =&\langle v^0_t\widetilde{\pi_0}(\langle  y,x\rangle)v^0_{t^{-1}}Vu_t b, Vu_t b'\rangle=\langle \widetilde{\pi_0}(\langle  y,x\rangle)V b,
V b'\rangle
\\ =& \langle \widetilde{\Psi_0}(x) V b, \widetilde{\Psi_0} (y) V b'\rangle.
\end{align*}  
Indeed, for fix $t\in G$, the continuity of the maps $t\mapsto \eta_t (x)$ and $t\mapsto u_t b$ for $b\in \m B,~x\in E$ provides the 
fact that the map $t\mapsto w^0_t(z)$ is continuous for each $z\in \widetilde{\Psi_0}(E)V\m B$. Therefore $w^0$ is a unitary representation
of $G$ on $[\widetilde{\Psi_0}(E)V\m B]$ and hence it naturally extends to a unitary representation of $G$ on the strong
operator topology closure of $[\widetilde{\Psi_0}(E)V\m B]$ in $\ma
B(\m H,F\bigodot \m H)$, which we denote by
$w$. 

Note that $E\otimes C_c (G)$ is dense in $E\times_{\eta} G$ (cf.
Theorem 3.5 of \cite{Jo12b}). For $x\in E$ and $f\in C_c(G)$ we have
\begin{eqnarray*}
(\widetilde{\Psi_0}\times v^0)(x\otimes f)&=&\int_G
\widetilde{\Psi_0} (x f(t))v^0_t dt =\int_G
\overline{\Psi_0}(i_{E}(x f(t))) \overline{\pi_0}(i_G (t))dt\\
&=& \overline{\Psi_0}(i_{E}(x)\int_G f(t) i_G (t)dt
)=\overline{\Psi_0}(i_{E}(y)i_{\m A}(\langle y,y \rangle)\int_G f(t)
i_G (t)dt)\\ &=& \overline{\Psi_0} (i_{E}(y)(\langle
y,y\rangle\otimes f))=\overline{\Psi_0}(y\langle y,y\rangle\otimes
f)={\Psi_0}(x\otimes f)
\end{eqnarray*}
where $x=y\langle y,y\rangle~\mbox{for some}~y\in E~(\mbox{cf.
Proposition 2.31 \cite{RW98}})$. Also the 3rd last equality follows
from Corollary 2.36 of \cite{Wi07}. This proves $\widetilde{\Psi_0}\times
v^0=\Psi_0$ on $E\times_{\eta} G$. Also for all $m\in
C_c(G,E)$ and $b\in \m B$ we get
\begin{eqnarray*}
 S u'_t (T(m)b) &=& S T(\eta_t \circ m^l_t)b=SS^*\Psi_0 (\eta_t\circ m^l_t)Vb=\Psi_0 (\eta_t\circ m^l_t)Vb
 \\ &=&\int_G \widetilde{\Psi_0}(\eta_t (m(t^{-1}s))) v^0_s V b ds
= w_t\int_G \widetilde{\Psi_0}(m(t^{-1}s)) v^0_{t^{-1}s}V b ds \\ &=& w_t \Psi_0 (m) Vb=w_t ST(m)b.
\end{eqnarray*}
As $T$ is $(u',u)$-covariant, it satisfies $T(\eta_t \circ m^{l}_t)=u'_t
 T(m)$, where
 $m^{l}_t(s)=m(t^{-1} s)$ for all $s,t\in G$, $m\in C_c(G,E)$. Thus the strong operator topology closure of
$[T(E\times_{\eta} G)\m B]$ in $\ma B(\m H, F \bigodot \m H)$, say $F_T$, is invariant under $u'$. This together with the fact that $S$ is
an orthogonal projection onto $F_T$ provides $Su'_tz=w_tSz$ for all $z\in F^{\perp}_T$. 
So we obtain the equality $S u'_t y=w_t S y~\mbox{for all}~y\in  F.$
Hence
\begin{eqnarray*}
 T_0(\eta_t(x))b=S^*\widetilde{\Psi_0}(\eta_t(x))V b=S^* w_t \widetilde{\Psi_0}(x)V  u_{t^{-1}}b=u'_t T_0(x)u^*_t b
\end{eqnarray*}
for all $t\in G$, $x\in E$ and $b\in \m B$. Moreover,
\begin{eqnarray*}
\widetilde{T_0}(m)b&=&S^*\int_G
\widetilde{\Psi_0}(m(t))Vu_t
bdt=S^*\Psi_0(m)Vb=T(m)b
\end{eqnarray*}
for all
$m\in C_c (G,E)$, $b\in \m B$. This gives $\widetilde{T_0}=T$ and proves that the map $\mf{I}$ is onto.

Let $\tau_1:\m A\to\m B$ be a completely positive map and let $T_1:E\to  F$ be a $(u',u)$-covariant $\tau_1$-map satisfying
$\widetilde{T}_1=T$. If $(\Psi_1,\pi_1,V_1,S_1,F_1, F'_1)$ is the $(w_1,v_1)$-covariant Stinespring type construction associated to 
$T_1$ coming from Theorem \ref{prop1.6}, then
we show that $(\Psi_1\times v_1, V_1, S_1,F_1, F'_1)$ is unitarily equivalent to
the Stinespring type construction associated to $T$. Indeed, from Proposition \ref{prop3.3}, 
there exists a decomposition $$\widetilde{T_1}(m)=S^*_1 (\Psi_1\times v_1)(m)V_1~\mbox{for all}~m\in C_c(G,E).$$  
This implies that for all $m,m'\in C_c(G,E)$ we get
\begin{align*}
 \tau(\langle m,m'\rangle)&=\langle T(m),T(m')\rangle=\langle \widetilde{T_1}(m),\widetilde{T_1}(m')\rangle
\\ & =\langle S^*_1 (\Psi\times v_1)(m)V_1,S^*_1 (\Psi\times v_1)(m')V_1\rangle
\\ & =\langle (\pi_1\times v_1)(\langle m,m'\rangle)V_1,V_1\rangle.
 \end{align*}
 $E$ is full gives $E\times_{\eta} G$ is full (cf. the proof of Proposition 3.5, \cite{EKQR00}) and 
hence $\tau(f)=\langle (\pi_1\times v_1)(f)V_1,V_1\rangle$ for all $f\in C_c (G,\m A)$. Using this fact we deduce that 

\begin{align*}
 \langle \pi(f)Vb,\pi(f')Vb'\rangle &=\langle \pi(f^{\prime*}f)Vb,Vb'\rangle=b^*\tau(f^{\prime*}f) b'\\ &=\langle \pi_1\times v_1(f)V_1b,
 \pi_1\times v_1(f')V_1b'\rangle
 \end{align*}
for all $f,f'\in C_c(G,\m A)$ and $b,b'\in \m B$. Thus we get a unitary $U_1:F\to F_1$ defined by
\[
U_1(\pi(f)V b):=\pi_1\times v_1 (f)V_1b ~\mbox{for}~f \in
C_c(G,\m A),~b\in\m B
\]
and which satisfies $V_1=U_1 V$, $\pi_1\times v_1 (f)=U_1\pi(f)U^*_1$ for all $f \in
C_c(G,\m A)$. Another computation 

\begin{align*}
 \| \Psi(m)Vb\|^2 &=\|\langle\Psi(m)Vb,\Psi(m)Vb\rangle\|=\|\langle\pi(\langle m,m\rangle)Vb,Vb\rangle\|=\|b^*\tau(\langle m,m\rangle) b\|
 \\ &= \|b^*\langle \pi_1\times v_1(\langle m,m\rangle)V_1,V_1\rangle b\|=\|\langle\Psi_1\times v_1(m)V_1b,\Psi_1\times v_1(m)V_1b\rangle\|
 \\ &=\| \Psi_1\times v_1(m)V_1b\|^2
\end{align*}
for all $m\in C_c (G,E)$, $b\in \m B$ provides a unitary $U_2:  F'\to F'_1$ defined as
\[
 U_2 (\Psi(m)Vb):=\Psi_1\times v_1 (m)V_1 b~\mbox{for}~m \in C_c(G,E),~b\in\m B.
\]
Further, it satisfies conditions $S_1=U_2S$ and $U_2\Psi(m)=\Psi_1\times v_1(m)U_1$ for all $m\in C_c(G,E)$. This implies 
$U_2\overline{\widetilde{\Psi} \times v}(z')=\overline{\Psi_1 \times v_1}(z')U_1$ for all $z' \in
M(E\times_{\eta} G)$ and so $U_2\widetilde{\Psi}(x)=\Psi_1\times v_1 (x)U_1$ for
all $x\in E$. Using it we have
\begin{eqnarray*}
 T_0(x)=S^*\widetilde{\Psi}(x)V=S^*_1 U_2\widetilde{\Psi}(x)U^*_1V_1 =S^*_1 U_2 U^*_2 (\Psi_1\times v_1)(x)U_1U^*_1V_1  =T_1(x) 
\end{eqnarray*}
for all $x\in E$ and $b\in\m B$. Hence $\mf{I}$ is injective.
\end{proof}

\noindent \textbf{Acknowledgement.} The author would like to
express thanks of gratitude to Santanu Dey for several
discussions. This work was supported by CSIR, India.

\bibliographystyle{amsplain}{
\bibliography{harshbib}}

\providecommand{\bysame}{\leavevmode\hbox to3em{\hrulefill}\thinspace}
\providecommand{\MR}{\relax\ifhmode\unskip\space\fi MR }
\providecommand{\MRhref}[2]{%
  \href{http://www.ams.org/mathscinet-getitem?mr=#1}{#2}
}
\providecommand{\href}[2]{#2}
\begin{thebibliography}{10}

\bibitem{BRS12}
B.~V.~Rajarama Bhat, G.~Ramesh, and K.~Sumesh, \emph{Stinespring's theorem for
  maps on {H}ilbert {$C^\ast$}-modules}, J. Operator Theory \textbf{68} (2012),
  no.~1, 173--178. \MR{2966040}

\bibitem{BSu13}
B.~V.~Rajarama Bhat and K.~Sumesh, \emph{Bures distance for completely positive
  maps}, Infin. Dimens. Anal. Quantum Probab. Relat. Top. \textbf{16} (2013),
  no.~4, 1350031, 22. \MR{3192708}

\bibitem{BMSS12}
Panchugopal Bikram, Kunal Mukherjee, R.~Srinivasan, and V.~S. Sunder,
  \emph{Hilbert von {N}eumann modules}, Commun. Stoch. Anal. \textbf{6} (2012),
  no.~1, 49--64. \MR{2890849}

\bibitem{CE76}
Man~Duen Choi and Edward~G. Effros, \emph{The completely positive lifting
  problem for {$C\sp*$}-algebras}, Ann. of Math. (2) \textbf{104} (1976),
  no.~3, 585--609. \MR{0417795 (54 \#5843)}

\bibitem{DH14}
Santanu Dey and Harsh Trivedi, \emph{{$\mathfrak{K}$}-families and
  {CPD-H}-extendable families}, arXiv:1409.3655v1 (2014).

\bibitem{EKQR00}
Siegfried Echterhoff, S.~Kaliszewski, John Quigg, and Iain Raeburn,
  \emph{Naturality and induced representations}, Bull. Austral. Math. Soc.
  \textbf{61} (2000), no.~3, 415--438. \MR{1762638 (2001j:46101)}

\bibitem{He99}
Jaeseong Heo, \emph{Completely multi-positive linear maps and representations
  on {H}ilbert {$C^*$}-modules}, J. Operator Theory \textbf{41} (1999), no.~1,
  3--22. \MR{1675235 (2000a:46103)}

\bibitem{HJ11}
Jaeseong Heo and Un~Cig Ji, \emph{Quantum stochastic processes for maps on
  {H}ilbert {$C^*$}-modules}, J. Math. Phys. \textbf{52} (2011), no.~5, 053501,
  16. \MR{2839082 (2012h:81175)}

\bibitem{Jo11}
Maria Joi{\c{t}}a, \emph{Covariant version of the {S}tinespring type theorem
  for {H}ilbert {$C^*$}-modules}, Cent. Eur. J. Math. \textbf{9} (2011), no.~4,
  803--813. \MR{2805314 (2012f:46110)}

\bibitem{Jo12b}
\bysame, \emph{Covariant representations of {H}ilbert {$C^*$}-modules}, Expo.
  Math. \textbf{30} (2012), no.~2, 209--220. \MR{2928201}

\bibitem{Kap93}
Alexander Kaplan, \emph{Covariant completely positive maps and liftings}, Rocky
  Mountain J. Math. \textbf{23} (1993), no.~3, 939--946. \MR{1245456
  (94k:46139)}

\bibitem{La95}
E.~C. Lance, \emph{Hilbert {$C^*$}-modules}, London Mathematical Society
  Lecture Note Series, vol. 210, Cambridge University Press, Cambridge, 1995, A
  toolkit for operator algebraists. \MR{1325694 (96k:46100)}

\bibitem{Pas73}
William~L. Paschke, \emph{Inner product modules over {$B^{\ast} $}-algebras},
  Trans. Amer. Math. Soc. \textbf{182} (1973), 443--468. \MR{0355613 (50
  \#8087)}

\bibitem{Pau82}
Vern Paulsen, \emph{A covariant version of {E}xt}, Michigan Math. J.
  \textbf{29} (1982), no.~2, 131--142. \MR{654474 (83f:46076)}

\bibitem{RT03}
Iain Raeburn and Shaun~J. Thompson, \emph{Countably generated {H}ilbert
  modules, the {K}asparov stabilisation theorem, and frames with {H}ilbert
  modules}, Proc. Amer. Math. Soc. \textbf{131} (2003), no.~5, 1557--1564
  (electronic). \MR{1949886 (2003j:46089)}

\bibitem{RW98}
Iain Raeburn and Dana~P. Williams, \emph{Morita equivalence and
  continuous-trace {$C^*$}-algebras}, Mathematical Surveys and Monographs,
  vol.~60, American Mathematical Society, Providence, RI, 1998. \MR{1634408
  (2000c:46108)}

\bibitem{Ri74}
Marc~A. Rieffel, \emph{Induced representations of {$C^{\ast} $}-algebras},
  Advances in Math. \textbf{13} (1974), 176--257. \MR{0353003 (50 \#5489)}

\bibitem{Sk00}
Michael Skeide, \emph{Generalised matrix {$C^*$}-algebras and representations
  of {H}ilbert modules}, Math. Proc. R. Ir. Acad. \textbf{100A} (2000), no.~1,
  11--38. \MR{1882195 (2002k:46155)}

\bibitem{Sk01}
\bysame, \emph{Hilbert modules and applications in quantum probability},
  Habilitationsschrift (2001).

\bibitem{Ske12}
\bysame, \emph{Hilbert von {N}eumann modules versus concrete von {N}eumann
  modules}, arXiv:1205.6413v1 (2012).

\bibitem{SSu14}
Michael Skeide and K.~Sumesh, \emph{C{P}-{H}-extendable maps between {H}ilbert
  modules and {CPH}-semigroups}, J. Math. Anal. Appl. \textbf{414} (2014),
  no.~2, 886--913. \MR{3168002}

\bibitem{St55}
W.~Forrest Stinespring, \emph{Positive functions on {$C^*$}-algebras}, Proc.
  Amer. Math. Soc. \textbf{6} (1955), 211--216. \MR{0069403 (16,1033b)}

\bibitem{Wi07}
Dana~P. Williams, \emph{Crossed products of {$C{^\ast}$}-algebras},
  Mathematical Surveys and Monographs, vol. 134, American Mathematical Society,
  Providence, RI, 2007. \MR{2288954 (2007m:46003)}

\end{thebibliography}

{\footnotesize Department of Mathematics, Indian Institute of Technology Bombay,}

{\footnotesize Powai, Mumbai-400076,}

{\footnotesize India.}

{\footnotesize e-mail: harsh@math.iitb.ac.in}

\end{document}